\documentclass[11pt, a4paper,DIV=11]{scrartcl}

\usepackage{amsthm,amsfonts,amsmath,latexsym,mathrsfs,amssymb}
\usepackage{mathtools}

\usepackage[utf8]{inputenc}
\usepackage{xcolor}

\usepackage{libertine}
\usepackage{libertinust1math}
\usepackage[T1]{fontenc}

\usepackage{hyperref}
\hypersetup{colorlinks=true, linkcolor=blue!60!black, citecolor=blue!60!black, urlcolor=blue!60!black}

\newtheorem{theorem}{Theorem}[section]
\newtheorem{lemma}[theorem]{Lemma}
\newtheorem{result}[theorem]{Result}

\newtheorem{corollary}[theorem]{Corollary}

\theoremstyle{definition}

\theoremstyle{remark}

\numberwithin{equation}{section}

\newcommand{\eps}{\varepsilon}

\newcommand{\OM}{\Omega}
\newcommand{\disk}{\mathbb{D}}
\newcommand{\Cp}{\mathbb{C}_{01}}
\newcommand{\Chat}{\widehat{\mathbb C}}

\newcommand{\hol}{\mathcal{O}}

\newcommand{\C}{\mathbb{C}}

\newcommand{\N}{\mathbb{N}}

\newcommand{\D}{\mathbb{D}}

\usepackage{enumitem}

\setlist[enumerate]{
  itemsep=2pt,
  topsep=2pt,
  parsep=0pt,
  partopsep=0pt,
}

\makeatletter
\def\blfootnote{\gdef\@thefnmark{}\@footnotetext}
\makeatother

\begin{document}

% \title
\title{Uniformization of domains in the Riemann sphere via the Kobayashi metric}

\author{Jaikrishnan Janardhanan \thanks{The author is supported by a grant from ANRF: \textbf{ANRF/ARGM/2025/000619/MTR} \\
\href{mailto:jaikrishnan@iitpkd.ac.in}{jaikrishnan@iitpkd.ac.in}\\
Department of Mathematics, Indian Institute of Technology Palakkad,
Palakkad, Kerala-678623, India} }

\date{\today}

\maketitle

\begin{abstract}
  We prove the uniformization theorem for domains in the
  Riemann sphere via the Kobayashi metric. Our proof does not rely on the
  existence of the elliptic modular function and instead uses the
  characterization of complete Kobayashi hyperbolicity for domains in the
  sphere. 
\end{abstract}

\blfootnote{\textup{2020} \textit{Mathematics Subject Classification}:
{Primary: 30F45; Secondary: 30F10, 32Q45}}


\section{Introduction}

Ahlfors, in his seminal work \cite{Ahlfors1938}, made the crucial connection
between curvature and complex analysis by showing that the classical
Schwarz--Pick lemma can be interpreted as a statement about negative curvature.
Broadly speaking, Ahlfors showed that a holomorphic map from the unit disk to a
Riemann surface equipped with a metric of Gaussian curvature at most $-4$ is
distance-decreasing. This result is now known as Ahlfors' lemma, or the
Ahlfors--Schwarz lemma. Ahlfors used this lemma to give elementary proofs of
several classical theorems, including the theorems of Picard, Schottky and
Bloch, without recourse to the elliptic modular function; see
\cite[Chapter~1]{Ahlfors1960ConformalInvariants}. Since then, Ahlfors' lemma and its
variants have become central tools in complex analysis and geometry; see
\cite{Osserman1999Notices,kim2011schwarz}.

Inspired by Ahlfors' lemma and its various generalizations, Kobayashi \cite{Kobayashi1967} introduced a pseudodistance, now known as the
Kobayashi pseudodistance, on complex manifolds. Holomorphic maps are
distance-decreasing for the Kobayashi pseudodistance, just as holomorphic maps
of the disk are distance-decreasing for the Poincar\'e distance. A complex
manifold is called \emph{Kobayashi hyperbolic} when this pseudodistance is a
genuine distance. In the context of Riemann surfaces, which is the original
setting of Ahlfors' lemma, hyperbolicity in the sense of Kobayashi coincides
with the classical notion: a Riemann surface is Kobayashi hyperbolic if and only
if its universal covering surface is the unit disk. This follows from the
uniformization theorem and the fact that Kobayashi hyperbolicity is preserved
under covering maps. To the best of our knowledge, all known proofs of this
fact use the uniformization theorem.

For domains in the sphere, however, one can prove the equivalence of Kobayashi
hyperbolicity and classical hyperbolicity using the much simpler special case of
the uniformization theorem for domains in the Riemann sphere. This result seems
to go back to Poincar\'e. Most modern proofs, such as the one in
\cite[Chapter~13]{Zakeri2021}, begin by first constructing the elliptic modular
function and thereby showing that the universal covering of $\Cp := \C \setminus
\{0,1\}$ is the unit disk. Having constructed the elliptic modular function, the strategy is then
modeled on the standard proof of the Riemann mapping theorem: one considers an
appropriate family of holomorphic functions whose extremal member is the
required covering map. A typical formulation fixes $\alpha\in \OM$ (here $\OM$
is a domain in $\Cp$) and considers the family of all
holomorphic covering maps $f:\D\to\Cp$ with $f(0)=\alpha$ for which a suitable restriction
of $f$ is a covering map onto $\OM$. As is apparent, the family involved in this
extremal argument is quite complicated, and the proof that the extremal map is
the required covering map is technical. In particular, one needs the
elliptic modular function to even show that this family is nonempty.

In \cite{Bharathi2026kob}, we have shown, without relying on any form of
uniformization, that a domain in the sphere is complete Kobayashi hyperbolic iff
its complement contains at least
three points. The proof is elementary and
uses ideas from the rescaling technique due to Zalcman. The purpose of the present paper is to use
this characterization to prove the uniformization theorem for domains in the
sphere. Apart from being of intrinsic interest, our proof is also less technical than the
aforementioned proof of the uniformization theorem using the elliptic modular function.
The main simplification lies in the family to which we apply an extremal
argument. Let $X$ be the universal cover of a bounded planar domain $\OM$. We
consider the family of holomorphic coverings from $X$ onto subdomains of the unit disk containing $0$ and
maximize the derivative at a fixed base point. The family of such maps is
nonempty by hypothesis, and an extremal map exists by Montel's theorem. It is straightforward
that this extremal map is a local biholomorphism. We then use the
properties of the Kobayashi--Royden metric to show that this extremal map is
again a covering. The
classical square-root trick, as in the proof of the Riemann mapping theorem,
forces the extremal image to be the whole disk. It follows that $X$ is
biholomorphic to $\D$ and this settles the case of bounded
domains. The general case then follows from a straightforward exhaustion
argument and Montel's theorem. 

Of course, one might argue that whatever simplification we have achieved is by importing
the machinery of Kobayashi hyperbolicity as a substitute for the elliptic
modular function. We regard this as a virtue rather than a defect. In
\cite{Bharathi2026kob}, we also showed that Kobayashi hyperbolicity and completeness are
equivalent to the classical theorems of Landau, Picard,
Schottky and Montel. This was first shown by Hahn \cite{Hahn1980}, but his proof
relies on the fact that the universal cover of $\Cp$ is the unit disk. 
However, once the completeness of the Kobayashi metric on $\Cp$ is established,
the equivalence follows by a short and simple argument; see Section~4 and Theorem~6.2 of
\cite{Bharathi2026kob}.
This suggests that Kobayashi hyperbolicity provides a
natural unifying framework in which to understand these theorems much like differential forms provide a natural unifying framework in which to
understand the classical theorems of Gauss, Green and Stokes.  
The present paper adds the uniformization theorem to the list of results that
can be understood via Kobayashi hyperbolicity.

There is another framework in which one can prove both the classical theorems of
Montel, Picard, Schottky and Landau as well as the uniformization theorem for domains
in the sphere without using the elliptic modular function.  This is Heins' \cite{Heins1962}
theory of $SK$-metrics which was directly inspired by Ahlfors' lemma and Ahlfors'
notion of ultrahyperbolic metrics (\cite[Definition~1-1]{Ahlfors1960ConformalInvariants}).
An $SK$-metric on a Riemann surface is an upper-semicontinuous conformal
pseudometric whose generalized Gaussian curvature is at most $-4$; see
\cite[Chapter~1]{Heins1962} for details. In particular,
any conformal metric with curvature bounded above by $-4$ is an $SK$-metric.
Heins showed that if a Riemann surface admits
an $SK$-metric, then there is a unique maximal $SK$-metric on the surface, and
furthermore this metric is complete, regular (in fact, real-analytic) and has
constant curvature $-4$. Grauert and Reckziegel \cite{Grauert1965hermitian}
have constructed a conformal metric on $\Cp$ with curvature bounded above by
$-4$ and therefore, by Heins' theory, $\Cp$ admits a complete conformal
metric of constant curvature $-4$. It is now not difficult to show that the
universal cover of $\Cp$ is the unit disk. A detailed modern exposition of all of these
facts can be found in the article by
Kraus and Roth \cite{RothKraus2013}. It is worthwhile highlighting some merits
of our approach in comparison to the $SK$-metric approach:
\begin{enumerate}
        \item The proof of the existence of the maximal $SK$-metric and its
        properties is quite technical and long. In contrast, our proof of the
        uniformization theorem is elementary and all the tools we use (for example,
        Montel's theorem) themselves admits elementary proofs using the
        Kobayashi metric. 
        \item While there are many constructions of $SK$-metrics on $\Cp$, none
        of them are entirely transparent and intuitive.
        \item The Kobayashi metric and Kobayashi hyperbolicity are now standard
        tools in several complex variables. Seeing these tools in action in the
        simple setting of planar domains has great pedagogical value.
\end{enumerate}

\subsection*{Notation.} \begin{enumerate}
        \item We write \(\hol(X,Y)\) for the set of holomorphic maps from a
        complex manifold \(X\) to a complex manifold \(Y\). When $Y = \C$, we simply write \(\hol(X)\) for \(\hol(X,\C)\).
        \item The unit disk in \(\C\) is denoted by \(\D\). The Riemann sphere
        is denoted by \(\Chat\). We denote the twice punctured plane $\C
        \setminus \{0,1\}$ by $\Cp$.
        \item We normalize the curvature of the Poincar\'e metric on the unit disk to be $-4$. With this normalization, the Poincar\'e metric on $\D$ is given by:
        \[
                \rho(z;v) := \frac{|v|}{1-|z|^2}.
        \]

        \item The Kobayashi pseudodistance and the Kobayashi--Royden
        pseudometric on a complex manifold \(X\) are denoted by \(d_X\) and
        \(F_X\), respectively.  We write \(\kappa_D\) for the Kobayashi--Royden
        density on a planar domain \(D\) defined by
        \[
                \kappa_D(z) := F_D(z;1).
        \]
        \item We denote the length of a piecewise \(C^1\) path \(\gamma\) in a complex manifold \(X\) with respect to the Kobayashi--Royden pseudometric by \(L_X(\gamma)\).
        \item All Riemann surfaces are assumed to be connected.

\end{enumerate}

\section{Technical necessities}

We refer the reader to our earlier paper \cite{Bharathi2026kob} for a quick
overview of the Kobayashi pseudodistance and the Kobayashi--Royden
pseudometric. More expansive treatments can be found in
\cite{Abate2023book,SKobayashi1998,jp2013}.

We shall use the well-known fact that the topology induced by the Kobayashi
distance on a hyperbolic complex manifold is the same as the manifold topology
without explicit mention. We state one major result about the
behaviour of the Kobayashi pseudodistance and Kobayashi--Royden pseudometric
under holomorphic covering maps, which will play a central role in the proof of the
uniformization theorem.

\begin{result}\label{R:cover-isometry}
Let \(\pi:X\to Y\) be a holomorphic covering map of complex manifolds.  Then
$\pi$ is a local isometry with respect to the Kobayashi--Royden pseudometrics
\(F_X\) and \(F_Y\) on \(X\) and \(Y\), respectively. More precisely, we have 
\[
  \pi^*F_Y = F_X,
\]
where $\pi^*F_Y (x;v) := F_Y(\pi(x); d\pi_x v)$ is the pullback of \(F_Y\) by \(\pi\). 
Furthermore,
\[
        d_Y(y_1,y_2)=\inf\{d_X(x_1,x_2):\pi(x_1)=y_1,\ \pi(x_2)=y_2\},
\]
where \(d_X\) and \(d_Y\) are the Kobayashi pseudodistances on \(X\) and \(Y\),
respectively. 
In particular, $Y$ is (complete) Kobayashi hyperbolic if and only if $X$ is
(complete) Kobayashi hyperbolic. 

\end{result}

We will use the following result from \cite{Bharathi2026kob} that characterizes
complete Kobayashi hyperbolicity for domains in the Riemann sphere. We emphasize
once again
that the proof of this result does not use any form of uniformization and is
elementary. 

\begin{result}[Theorem~6.2 of \cite{Bharathi2026kob}]\label{R:complete-hyperbolic}
  A domain in the Riemann sphere is Kobayashi hyperbolic if and only if its
  complement contains at least three points. Moreover, such domains are complete
  Kobayashi hyperbolic. 
\end{result}

Montel's theorem is a direct consequence of the above result. 

\begin{result}[Montel's theorem]
   Let $D \subset \Chat$ be a domain whose complement contains at least three
   points. Then $\hol(X,D)$ is a normal family, where $X$ is any complex manifold.      
\end{result}

We require the following basic lemma about the behaviour of the
Kobayashi--Royden pseudometric under an increasing union.

\begin{lemma}\label{L:exhaustion}
Let
\[
    D_1\subset D_2\subset\cdots\subset\C
\]
be domains whose union is a Kobayashi hyperbolic domain \(D\). Then
\[
    \kappa_{D_j}\longrightarrow \kappa_D
\]
uniformly on compact subsets of \(D\).
\end{lemma}

\begin{proof}
Fix \(z\in D\) and choose $N_0 \in \N$ such that $z \in D_j, j \geq N_0$. We have the obvious inequalities
\[ 
        \kappa_D(z)\leq \kappa_{D_{j+1}}(z)\leq \kappa_{D_j}(z).
\]
Thus \(\kappa_{D_j}(z)\) decreases to a limit \(L\geq\kappa_D(z)\). We need to show \(L\leq\kappa_D(z)\).

Let \(\eps>0\).  By the definition of \(\kappa_D(z)\), there is \(f\in\hol(\D,D)\) with \(f(0)=z\) such that
\[
        \frac1{|f'(0)|}<\kappa_D(z)+\eps.
\]
Fix \(0<r<1\). We may assume, by increasing $N_0$ if necessary, that the compact
set \(f(r\overline\D)\) is contained in \(D_j\) for all \(j \geq N_0\). Define
the holomorphic function $g_r:\D\to D_j$ by $g_r(\zeta) := f(r\zeta)$. 
We have $g_r(0) = z$ and $g_r'(0) = r f'(0)$ and therefore
\[
        \kappa_{D_j}(z)\leq \frac1{r|f'(0)|}
        <\frac{\kappa_D(z)+\eps}{r}.
\]
Letting \(j\to\infty\), then \(r\to1\), and finally \(\eps\to0\), gives
\(L\leq \kappa_D(z)\). Hence \(\kappa_{D_j}(z)\to\kappa_D(z)\) pointwise.
Dini's theorem gives uniform convergence on compact subsets of \(D\). 
\end{proof}

The following lemma is a special case of a famous theorem from Riemannian
geometry due to Ambrose; see \cite[Theorem~5.4]{Sakai1996} for the general version. We give a proof in the present setting for the sake of completeness.

\begin{lemma}\label{L:cover}
  Let \(X\) and \(Y\) be Kobayashi hyperbolic Riemann surfaces equipped with
  the Kobayashi--Royden metrics \(F_X\) and \(F_Y\), respectively. Let \(d_X\) and
  \(d_Y\) denote the corresponding distances. Assume that \((X,d_X)\) is
  complete and that $Y$ is simply connected. Let $f:X \to Y$ be a local
  biholomorphism that is also a local
  isometry with respect to \(F_X\) and \(F_Y\). Then $f$ is a biholomorphism.
\end{lemma}

\begin{proof}

Choose \(x_0\in X\), and set \(y_0 := f(x_0)\). Since \(f\) is a local
  biholomorphism, there is a neighbourhood \(V_0\) of \(y_0\) and a holomorphic
  inverse $g_0:V_0\to X$ of \(f\) defined on \(V_0\) such that $g_0(y_0)=x_0$. 
  We first show that \(g_0\) admits analytic continuation along every
  piecewise \(C^1\) path in \(Y\) starting at \(y_0\). Let
  \(\gamma:[0,1]\to Y\) be such a path. Define
  \[
    E := \{t\in[0,1]: \text{the germ of $g_0$ at $y_0$ can be analytically continued along $\gamma|_{[0,t]}$}\}.
  \] 
  This set is nonempty and open. We will show that \(E\) is also closed. Let \(t^* := \sup E\).  If \(t^*=1\), then we are done. So assume \(t^*<1\)  
  and suppose that the germ of $g_0$ at $y_0$ has
  been analytically continued along \(\gamma|_{[0,t^*)}\). Denote by $g_t$ the
  analytic continuation of $g_0$ along $\gamma|_{[0,t^*)}$ and define
  \[
        \widetilde\gamma(t) := g_t(\gamma(t)).
  \]
  Since \(f\) is a local isometry for the Kobayashi--Royden metrics, the
  lift has the same Kobayashi length as \(\gamma\). Indeed, if $s,t \in
  [0,t^*)$ with $s < t$, we have
  \[
          d_X(\widetilde\gamma(s),\widetilde\gamma(t))
        \leq L_X(\widetilde\gamma|_{[s,t]})
        = L_Y(\gamma|_{[s,t]}). 
  \]
  It follows that $\widetilde\gamma$ is
  uniformly continuous. As $X$ is complete, $\widetilde\gamma$ extends
  to a continuous path on $[0,t^*]$.  Let $x^* := \widetilde\gamma(t^*)$.  We
  have $f \circ \widetilde\gamma = \gamma$ on $[0,t^*]$. Since
  $f$ is a local biholomorphism, there is a neighbourhood \(V^*\) of \(y^* :=
  \gamma(t^*) = f(x^*)\) such that \(f\) admits a holomorphic inverse \(g^*:V^*\to X\)
  defined on \(V^*\) with \(g^*(y^*)=x^*\).  For \(t<t^*\) sufficiently close to \(t^*\), the point \(\gamma(t)\) lies in
\(V^*\). Both $g_t$ and \(g^*\) are local inverses of \(f\) taking values near
  \(x^*\) and it follows that the germ of $g_0$ at \(y_0\) can be analytically
  continued along $\gamma|_{[0,t^*]}$. Thus $E$ is closed and is therefore
  equal to $[0,1]$. Hence \(g_0\)
  admits analytic continuation along \(\gamma\). Since \(Y\) is simply connected, the monodromy theorem implies that
 \(g_0\) extends to a holomorphic map $g:Y\to X$ satisfying $f\circ
 g=\textsf{id}_Y$. By the identity theorem, $g\circ f=\textsf{id}_X$. Thus \(f\) is a biholomorphism.

\end{proof}

\section{The uniformization theorem for spherical domains}

\begin{theorem}[Uniformization of domains in the Riemann sphere]
\label{thm:sphere-uniformization-streamlined}
Let \(\Omega\subset\Chat\) be a domain. Then the universal covering surface of
\(\Omega\) is biholomorphic to exactly one of $\Chat,\C$ or $\D$. More precisely:
\begin{enumerate}[label=\textup{(\roman*)}]
    \item if \(\Omega=\Chat\), then the universal cover is \(\Chat\);
    \item if \(\Chat\setminus\Omega\) consists of one point, then
    \(\Omega\cong\C\), and the universal cover is \(\C\);
    \item if \(\Chat\setminus\Omega\) consists of two points, then
    \(\Omega\cong\C \setminus \{0\}\), and the universal cover is \(\C\);
    \item if \(\Chat\setminus\Omega\) contains at least three points, then
    the universal cover is \(\D\).
\end{enumerate}
\end{theorem}

\begin{proof}
  The only non-trivial case is when \(\Chat\setminus\Omega\) contains at least
  three points. In this case, \(\Omega\) is complete
  Kobayashi hyperbolic by Result~\ref{R:complete-hyperbolic}. We may assume that \(\Omega\) does not contain the
  point at infinity. Let \(X\) be the universal cover of \(\Omega\). We will
  show that \(X\) is biholomorphic to \(\D\). We first consider the case when \(\Omega\) is
  bounded. We will exhibit the required biholomorphism as an extremal map in a
  suitable family of holomorphic maps from \(X\) to \(\D\).  
  Fix \(z_0\in X\).  It is evident that the family
\[
        \mathcal F := \left\{f\in\hol(X,\D): f(z_0)=0,
        \ f:X\to f(X) \text{ is a holomorphic covering}\right\}
\]
is nonempty.

Choose a coordinate chart around $z_0$, say \(\chi:U\to\D_R\), with \(\chi(z_0)=0\).  Define
\[
        M := \sup_{f\in\mathcal F}\left|(f\circ\chi^{-1})'(0)\right|.
\]
It is clear that $M > 0$. We will show that \(M<\infty\). 
Choose \(f_n\in\mathcal F\) with
\[
        \left|(f_n\circ\chi^{-1})'(0)\right|\to M.
\]
By Montel's theorem, a subsequence converges locally uniformly to a holomorphic
map \(g:X\to\overline\D\). This shows that $M$ cannot be infinite.  Since
\[
        \left|(g\circ\chi^{-1})'(0)\right|=M>0,
\]
the map \(g\) is nonconstant. As $g(z_0) = 0$, it follows that \(g(X)\subset\D\).

We next prove that \(g\) is locally biholomorphic.  In any local coordinate
\(\psi\), the functions \((f_n\circ\psi^{-1})'\) are zero-free and converge
locally uniformly to \((g\circ\psi^{-1})'\).  Hurwitz's theorem implies that the
latter is either zero-free or identically zero.  The latter would
force \(g\) to be constant on \(X\), contradicting \((g\circ\chi^{-1})'(0)\neq0\).
Thus \(g\) is locally biholomorphic.

Set $D_n := f_n(X)$ and $D := g(X)$.
We claim that any compact subset \(K\) of \(D\) is contained in \(D_n\) eventually.
Suppose not.  Then there is a subsequence \(n_j\) and points \(y_j\in K\) such that \(y_j\notin D_{n_j}\).  Passing to a further subsequence, assume \(y_j\to y\in K\).  Since \(y\in D\), choose \(a\in X\) with \(g(a)=y\).  The functions
$h_j := f_{n_j}-y_j$ are zero-free on \(X\), and \(h_j\to g-y\) uniformly on compact sets.  By Hurwitz's
theorem, \(g-y\) is either zero-free or identically zero.  But \((g-y)(a)=0\),
while \(g-y\not\equiv0\) because \(g\) is nonconstant.  This contradiction
proves the claim.

We now claim that $g$ is a local isometry, i.e., $g^*F_D=F_X$.
Contractivity gives
\[
        F_D(g(p);dg_p\xi)\leq F_X(p;\xi).
\]
For the reverse inequality, choose a relatively compact domain \(G\Subset D\)
with \(g(p)\in G\).  For all sufficiently large \(n\), we have \(f_n(p)\in G\)
and \(\overline G\subset D_n\).  By monotonicity of the Kobayashi metric under
inclusions and by Result~\ref{R:cover-isometry}, 
\[
        F_X(p;\xi)
        =F_{D_n}(f_n(p);df_n(p)\xi)
        \leq F_G(f_n(p);df_n(p)\xi).
\]
Letting \(n\to\infty\), and using the continuity of the Kobayashi metric on the relatively compact hyperbolic domain \(G\), gives
\[
        F_X(p;\xi)
        \leq F_G(g(p);dg_p\xi).
\]
We now exhaust $D$ by relatively compact domains \(G_i \Subset D\) with $G_1 \subset G_2 \subset \cdots$. Lemma~\ref{L:exhaustion} gives \(F_{G_i}\downarrow F_D\).  Hence
\[
        F_X(p;\xi)\leq F_D(g(p);dg_p\xi).
\]
The two inequalities prove the claim. Thus \(g\) is a local isometry with
respect to the Kobayashi--Royden metrics. Let $Y$ be the universal cover of $D$.
Since \(X\) is simply connected, \(g:X\to D\) lifts to a map
\(\widetilde g:X\to Y\). This lift is again a local biholomorphism and a
local isometry for the Kobayashi--Royden metrics. By Lemma~\ref{L:cover},
\(\widetilde g\) is a biholomorphism. Hence \(g:X\to D\) is a holomorphic
covering map and therefore $g \in \mathcal F$. 
The standard square-root trick shows that $D =\disk$ (see \cite[p.
231]{Ahlfors2021}) and as $g:X \to \disk$ is a holomorphic covering map, it follows that $X$ is biholomorphic to $\disk$. This settles the case of bounded domains.

For the unbounded case, let \(p\in\Omega\), and choose an exhaustion
\[
        D_1\Subset D_2\Subset\cdots\Subset\Omega,
        \qquad p\in D_1,
        \qquad \bigcup_nD_n=\Omega,
\]
by bounded domains.  We have holomorphic coverings
\[
        g_n:\D\to D_n,
        \qquad g_n(0)=p.
\]
By Montel's theorem, the maps form a normal family. As $g_n(0) = p$ for all $n$,
any subsequential limit, say $g$, must map into $\OM$. Moreover,
using \(g_n^*F_{D_n}=F_{\D}\) and Lemma~\ref{L:exhaustion}, we have
\[
        |g_n'(0)|=\kappa_{D_n}(p)^{-1}\longrightarrow \kappa_\Omega(p)^{-1}>0.
\]
Hence \(g\) is nonconstant. Once again,
Hurwitz's theorem applied to the derivatives shows that \(g\) is a local
biholomorphism. Repeating the compact-exhaustion argument used in the bounded
case shows that $g:\D \to \OM$ is a local isometry with respect to the
Kobayashi--Royden metrics. Let $Y$ be the universal cover of $\OM$ and let
$h:\D\to Y$ be a lift of $g:\D \to \OM$. By Result~\ref{R:cover-isometry}, it is
evident that $h$ is a local isometry for the Kobayashi--Royden metrics.
Lemma~\ref{L:cover} now implies that \(h\) is a
biholomorphism. Consequently the universal cover \(Y\) of
\(\Omega\) is biholomorphic to \(\D\). This proves the theorem.
\end{proof}

Combining the uniformization theorem for domains in the sphere with 
Theorem~6.2 from \cite{Bharathi2026kob} gives the following.

\begin{corollary}
  A domain in the Riemann sphere is Kobayashi hyperbolic if and only if its
  universal cover is the unit disk. Furthermore, such Kobayashi
  hyperbolic domains are automatically complete Kobayashi hyperbolic.
\end{corollary}

Once uniformization is established, we can exhibit the universal covering map 
as the extremal map that realizes the Kobayashi--Royden density at a given point.

\begin{corollary}
  Let \(\Omega\subset\Chat\) be a domain whose complement contains at least
  three points. Let $p \in \OM$ and consider the family
  \[
      \mathcal F := \{f \in \hol(\D,\OM): f(0) = p\}, \qquad M := \sup_{f \in \mathcal F} |f'(0)|.
  \]
  Then \(M<\infty\) and there is a holomorphic covering map $g \in \mathcal F$
  with \(g(0)=p\) such that \(|g'(0)|=M\). Consequently, $g:\D \to \OM$ is the
  universal covering map of $\OM$ with $g(0) = p$. 
\end{corollary}

\begin{proof}
  That $M < \infty$ follows from the fact that $\OM$ is Kobayashi hyperbolic and
  hence $\kappa_{\OM}(p) > 0$. By Montel's theorem, $\mathcal F$ is a normal
  family and therefore there is a subsequence of maps in $\mathcal F$ that
  converges uniformly on compact subsets to a holomorphic map $g:\D \to \OM$ with $g(0) =
  p$ and $|g'(0)| = M$. Let $\pi:\D \to \OM$ be the universal covering map with
  $\pi(0) = p$ and let $h:\D \to \D$ be a lift of $g:\D \to \OM$ by $\pi$ with $h(0) = 0$. As
  $\kappa_{\OM}(p) = 1/|\pi'(0)| = 1/|g'(0)|$, the Schwarz lemma implies that
  $h$  is an automorphism of $\D$. Hence $g = \pi \circ h$ is a holomorphic
  covering map.
\end{proof}

\subsection*{Disclosure on the use of AI tools.}
ChatGPT-5.5 Plus was used for drafting preliminary versions of some routine arguments, copyediting, and proofreading. GitHub Copilot was used for LaTeX autocompletion. 

 \bibliographystyle{amsalpha}

\bibliography{Kob_hyperbolic}

@book{Abate2023book,
  AUTHOR        = {Abate, Marco},
  TITLE         = {Holomorphic dynamics on hyperbolic {Riemann} surfaces},
  SERIES        = {De Gruyter Studies in Mathematics},
  VOLUME        = {89},
  PUBLISHER     = {De Gruyter, Berlin},
  YEAR          = {2023},
  PAGES         = {xiii+356},
  ISBN          = {978-3-11-060197-8; 978-3-11-060105-3; 978-059874-2},
  MRCLASS       = {37-02 (30C80 30F45 30J99 37F44 37F99)},
}

@book {Sakai1996,
    AUTHOR = {Sakai, Takashi},
     TITLE = {Riemannian geometry},
    SERIES = {Translations of Mathematical Monographs},
    VOLUME = {149},
      NOTE = {Translated from the 1992 Japanese original by the author},
 PUBLISHER = {American Mathematical Society, Providence, RI},
      YEAR = {1996},
     PAGES = {xiv+358},
      ISBN = {0-8218-0284-4},
   MRCLASS = {53-01 (53-02)},
MRREVIEWER = {Conrad\ Plaut},
       DOI = {10.1090/mmono/149},
       URL = {https://doi.org/10.1090/mmono/149},
}

@book{Ahlfors2021,
  AUTHOR        = {Ahlfors, Lars V.},
  TITLE         = {Complex analysis---an introduction to the theory of
                   analytic functions of one complex variable},
  EDITION       = {Third},
  NOTE          = {Reprint of the 1978 original},
  PUBLISHER     = {AMS Chelsea Publishing, Providence, RI},
  YEAR          = {2021},
  PAGES         = {xiv+331},
  ISBN          = {978-1-4704-6767-8},
  MRCLASS       = {30-01},
}

@book{Zakeri2021,
  AUTHOR        = {Zakeri, Saeed},
  TITLE         = {A course in complex analysis},
  PUBLISHER     = {Princeton University Press, Princeton, NJ},
  YEAR          = {2021},
  PAGES         = {xii+428},
  ISBN          = {9780691207582; 9780691218502},
  MRCLASS       = {30-01},
}

@misc{bharathi2026kob,
  author        = {Thiruvengadam, Bharathi and Janardhanan, Jaikrishnan},
  title         = {Revisiting {Kobayashi} hyperbolicity on planar domains},
  year          = {2026},
  howpublished  = {Preprint, {arXiv}:2604.19095 [math.{CV}]},
  eprint        = {2604.19095},
  archivePrefix = {arXiv},
  primaryClass  = {math.CV},
  url           = {https://arxiv.org/abs/2604.19095},
}

@article{Grauert1965hermitian,
  AUTHOR        = {Grauert, Hans and Reckziegel, Helmut},
  TITLE         = {Hermitesche {Metriken} und normale {Familien} holomorpher
                   {Abbildungen}},
  JOURNAL       = {Math. Z.},
  FJOURNAL      = {Mathematische Zeitschrift},
  VOLUME        = {89},
  YEAR          = {1965},
  PAGES         = {108--125},
  ISSN          = {0025-5874,1432-1823},
  MRCLASS       = {32.60 (53.80)},
  MRNUMBER      = {194617},
  MRREVIEWER    = {R.\ C.\ Gunning},
  DOI           = {10.1007/BF01111588},
  URL           = {https://doi.org/10.1007/BF01111588},
}

@book{kim2011schwarz,
  AUTHOR        = {Kim, Kang-Tae and Lee, Hanjin},
  TITLE         = {{Schwarz}'s lemma from a differential geometric viewpoint},
  SERIES        = {{IISc} Lecture Notes Series},
  VOLUME        = {2},
  PUBLISHER     = {IISc Press, Bangalore; World Scientific Publishing Co. Pte.
                   Ltd., Hackensack, NJ},
  YEAR          = {2011},
  PAGES         = {xvi+82},
  ISBN          = {978-981-4324-78-6; 981-4324-78-7},
  MRCLASS       = {30-02 (30C80 30F45 32H25 32Q45 53C20 53C55)},
  MRREVIEWER    = {Pavel\ A.\ Gumenuk},
}

@article {Heins1962,
    AUTHOR = {Heins, Maurice},
     TITLE = {On a class of conformal metrics},
   JOURNAL = {Nagoya Math. J.},
  FJOURNAL = {Nagoya Mathematical Journal},
    VOLUME = {21},
      YEAR = {1962},
     PAGES = {1--60},
      ISSN = {0027-7630,2152-6842},
   MRCLASS = {30.45},
MRREVIEWER = {B.\ Rodin},
       URL = {http://projecteuclid.org/euclid.nmj/1118801041},
}

@book{Ahlfors1960ConformalInvariants,
  AUTHOR        = {Ahlfors, Lars V.},
  TITLE         = {Conformal invariants: topics in geometric function theory},
  SERIES        = {McGraw-Hill Series in Higher Mathematics},
  PUBLISHER     = {McGraw-Hill Book Co., New York-D\"usseldorf-Johannesburg},
  YEAR          = {1973},
  PAGES         = {ix+157},
  MRCLASS       = {30-02},
  MRREVIEWER    = {M.\ H.\ Heins},
}

@article{Hahn1980,
  AUTHOR        = {Hahn, Kyong T.},
  TITLE         = {Equivalence of the classical theorems of {Schottky},
                   {Landau}, {Picard} and hyperbolicity},
  JOURNAL       = {Proc. Amer. Math. Soc.},
  FJOURNAL      = {Proceedings of the American Mathematical Society},
  VOLUME        = {89},
  YEAR          = {1983},
  NUMBER        = {4},
  PAGES         = {628--632},
  ISSN          = {0002-9939,1088-6826},
  MRCLASS       = {30F10 (30C99 32H15)},
  MRNUMBER      = {718986},
  MRREVIEWER    = {Toshio\ Urata},
  DOI           = {10.2307/2044595},
  URL           = {https://doi.org/10.2307/2044595},
}

@article{Ahlfors1938,
  AUTHOR        = {Ahlfors, Lars V.},
  TITLE         = {An extension of {Schwarz}'s lemma},
  JOURNAL       = {Trans. Amer. Math. Soc.},
  FJOURNAL      = {Transactions of the American Mathematical Society},
  VOLUME        = {43},
  YEAR          = {1938},
  NUMBER        = {3},
  PAGES         = {359--364},
  ISSN          = {0002-9947,1088-6850},
  MRCLASS       = {30C80},
  DOI           = {10.2307/1990065},
  URL           = {https://doi.org/10.2307/1990065},
}

@article{Kobayashi1967,
  AUTHOR        = {Kobayashi, Shoshichi},
  TITLE         = {Invariant distances on complex manifolds and holomorphic
                   mappings},
  JOURNAL       = {J. Math. Soc. Japan},
  FJOURNAL      = {Journal of the Mathematical Society of Japan},
  VOLUME        = {19},
  YEAR          = {1967},
  PAGES         = {460--480},
  ISSN          = {0025-5645,1881-1167},
  MRCLASS       = {57.60 (32.00)},
  DOI           = {10.2969/jmsj/01940460},
  URL           = {https://doi.org/10.2969/jmsj/01940460},
}

@article{Osserman1999Notices,
  AUTHOR        = {Osserman, Robert},
  TITLE         = {From {Schwarz} to {Pick} to {Ahlfors} and beyond},
  JOURNAL       = {Notices Amer. Math. Soc.},
  FJOURNAL      = {Notices of the American Mathematical Society},
  VOLUME        = {46},
  YEAR          = {1999},
  NUMBER        = {8},
  PAGES         = {868--873},
  ISSN          = {0002-9920,1088-9477},
  MRCLASS       = {30C80 (01A55 01A60 30-03 58D17)},
  MRREVIEWER    = {P.\ Lappan},
}

@book{SKobayashi1998,
  AUTHOR        = {Kobayashi, Shoshichi},
  TITLE         = {Hyperbolic complex spaces},
  SERIES        = {Grundlehren der mathematischen Wissenschaften [Fundamental
                   Principles of Mathematical Sciences]},
  VOLUME        = {318},
  PUBLISHER     = {Springer-Verlag, Berlin},
  YEAR          = {1998},
  PAGES         = {xiv+471},
  ISBN          = {3-540-63534-3},
  MRCLASS       = {32H20 (32H15 32H25 32H30)},
  MRREVIEWER    = {William\ A.\ Cherry},
  DOI           = {10.1007/978-3-662-03582-5},
  URL           = {https://doi.org/10.1007/978-3-662-03582-5},
}

@book{jp2013,
  AUTHOR        = {Jarnicki, Marek and Pflug, Peter},
  TITLE         = {Invariant distances and metrics in complex analysis},
  SERIES        = {De Gruyter Expositions in Mathematics},
  VOLUME        = {9},
  EDITION       = {extended},
  PUBLISHER     = {Walter de Gruyter GmbH \& Co. KG, Berlin},
  YEAR          = {2013},
  PAGES         = {xviii+861},
  ISBN          = {978-3-11-025043-5; 978-3-11-025386-3},
  MRCLASS       = {32-02 (32F45)},
  MRREVIEWER    = {Viorel\ V\^{a}j\^{a}itu},
  DOI           = {10.1515/9783110253863},
  URL           = {https://doi.org/10.1515/9783110253863},
}

@incollection {RothKraus2013,
    AUTHOR = {Kraus, Daniela and Roth, Oliver},
     TITLE = {Conformal metrics},
 BOOKTITLE = {Topics in modern function theory},
    SERIES = {Ramanujan Math. Soc. Lect. Notes Ser.},
    VOLUME = {19},
     PAGES = {41--83},
 PUBLISHER = {Ramanujan Math. Soc., Mysore},
      YEAR = {2013},
      ISBN = {978-93-80416-12-0},
   MRCLASS = {30F45 (30C80 35J91 53A30)},
MRREVIEWER = {William\ Ma},
}

\end{document}